\documentclass[a4paper, 11pt]{amsart}

\usepackage{graphics, graphicx, multicol}
\usepackage{color}
\usepackage{amssymb,amsbsy,amsmath,amsfonts,amssymb,amscd}
\usepackage{latexsym,euscript,epsfig}
\usepackage{pstricks}
\usepackage{array,delarray}
\usepackage{pst-node}
\usepackage{pspicture}
\usepackage[matrix,ps,xdvi]{xypic}

\usepackage{algorithm}
\usepackage{algorithmic}

\newtheorem{theorem}{Theorem}[section]

\newtheorem{proposition}{Proposition}[section]
\newtheorem{lemma}[theorem]{Lemma}
\newtheorem{definition}{Definition}[section]
\newtheorem{example}[theorem]{Example}

\newtheorem{remark}{Remark}[section]

\numberwithin{equation}{section}

   \newcommand{\Sn}{\mathfrak{S}_n}

    \newcommand{\knuthbtree}{\mathbb{P}}
    
    \newcommand{\knuthitree}{\mathbb{Q}}
    \newcommand{\fominPhypo}{\widehat{P}}
    \newcommand{\fominQhypo}{\widehat{Q}}
    \newcommand{\fominPSylv}{\widehat{\knuthbtree}}
    \newcommand{\fominQSylv}{\widehat{\knuthitree}}

    \def\to{\rightarrow}
    
    \def\egaldef{\stackrel{def}{=\,}}

    \def\permOneSymb{{\bf x}}
    \def\growthdiagramOfSigma{d(\sigma)}

\newcommand{\btone}{\bullet}
\newcommand{\btl}{{\begin{picture}(2,3)\put(1,1){\circle*{0.7}}\put(2,2){\circle*{0.7}}\put(2,2){\Line(-1,-1)}\put(2,2){\circle*{1}}\end{picture}}}
\newcommand{\btr}{{\begin{picture}(2,3)\put(1,2){\circle*{0.7}}\put(2,1){\circle*{0.7}}\put(1,2){\Line(1,-1)}\put(1,2){\circle*{1}}\end{picture}}}
\newcommand{\btll}{{\begin{picture}(3,4)\put(1,1){\circle*{0.7}}\put(2,2){\circle*{0.7}}\put(2,2){\Line(-1,-1)}\put(3,3){\circle*{0.7}}\put(3,3){\Line(-1,-1)}\put(3,3){\circle*{1}}\end{picture}}}
\newcommand{\btlr}{{\begin{picture}(3,4)\put(1,2){\circle*{0.7}}\put(2,1){\circle*{0.7}}\put(1,2){\Line(1,-1)}\put(3,3){\circle*{0.7}}\put(3,3){\Line(-2,-1)}\put(3,3){\circle*{1}}\end{picture}}}
\newcommand{\btlandr}{{\begin{picture}(3,3)\put(1,1){\circle*{0.7}}\put(2,2){\circle*{0.7}}\put(3,1){\circle*{0.7}}\put(2,2){\Line(-1,-1)}\put(2,2){\Line(1,-1)}\put(2,2){\circle*{1}}\end{picture}}}
\newcommand{\btrl}{{\begin{picture}(3,4)\put(1,3){\circle*{0.7}}\put(2,1){\circle*{0.7}}\put(3,2){\circle*{0.7}}\put(3,2){\Line(-1,-1)}\put(1,3){\Line(2,-1)}\put(1,3){\circle*{1}}\end{picture}}}
\newcommand{\btrr}{{\begin{picture}(3,4)\put(1,3){\circle*{0.7}}\put(2,2){\circle*{0.7}}\put(3,1){\circle*{0.7}}\put(2,2){\Line(1,-1)}\put(1,3){\Line(1,-1)}\put(1,3){\circle*{1}}\end{picture}}}
\newcommand{\btlll}{{\begin{picture}(4,5)\put(1,1){\circle*{0.7}}\put(2,2){\circle*{0.7}}\put(2,2){\Line(-1,-1)}\put(3,3){\circle*{0.7}}\put(3,3){\Line(-1,-1)}\put(4,4){\circle*{0.7}}\put(4,4){\Line(-1,-1)}\put(4,4){\circle*{1}}\end{picture}}}
\newcommand{\btllr}{{\begin{picture}(4,5)\put(1,2){\circle*{0.7}}\put(2,1){\circle*{0.7}}\put(1,2){\Line(1,-1)}\put(3,3){\circle*{0.7}}\put(3,3){\Line(-2,-1)}\put(4,4){\circle*{0.7}}\put(4,4){\Line(-1,-1)}\put(4,4){\circle*{1}}\end{picture}}}
\newcommand{\btlrl}{{\begin{picture}(4,5)\put(1,3){\circle*{0.7}}\put(2,1){\circle*{0.7}}\put(3,2){\circle*{0.7}}\put(3,2){\Line(-1,-1)}\put(1,3){\Line(2,-1)}\put(4,4){\circle*{0.7}}\put(4,4){\Line(-3,-1)}\put(4,4){\circle*{1}}\end{picture}}}
\newcommand{\btrll}{{\begin{picture}(4,5)\put(1,4){\circle*{0.7}}\put(2,1){\circle*{0.7}}\put(3,2){\circle*{0.7}}\put(3,2){\Line(-1,-1)}\put(4,3){\circle*{0.7}}\put(4,3){\Line(-1,-1)}\put(1,4){\Line(3,-1)}\put(1,4){\circle*{1}}\end{picture}}}
\newcommand{\btrlr}{{\begin{picture}(4,5)\put(1,4){\circle*{0.7}}\put(2,2){\circle*{0.7}}\put(3,1){\circle*{0.7}}\put(2,2){\Line(1,-1)}\put(4,3){\circle*{0.7}}\put(4,3){\Line(-2,-1)}\put(1,4){\Line(3,-1)}\put(1,4){\circle*{1}}\end{picture}}}
\newcommand{\btlandrr}{{\begin{picture}(4,4)\put(1,2){\circle*{0.7}}\put(2,3){\circle*{0.7}}\put(3,2){\circle*{0.7}}\put(4,1){\circle*{0.7}}\put(3,2){\Line(1,-1)}\put(2,3){\Line(-1,-1)}\put(2,3){\Line(1,-1)}\put(2,3){\circle*{1}}\end{picture}}}
\newcommand{\btrrl}{{\begin{picture}(4,5)\put(1,4){\circle*{0.7}}\put(2,3){\circle*{0.7}}\put(3,1){\circle*{0.7}}\put(4,2){\circle*{0.7}}\put(4,2){\Line(-1,-1)}\put(2,3){\Line(2,-1)}\put(1,4){\Line(1,-1)}\put(1,4){\circle*{1}}\end{picture}}}
\newcommand{\btrrr}{{\begin{picture}(4,5)\put(1,4){\circle*{0.7}}\put(2,3){\circle*{0.7}}\put(3,2){\circle*{0.7}}\put(4,1){\circle*{0.7}}\put(3,2){\Line(1,-1)}\put(2,3){\Line(1,-1)}\put(1,4){\Line(1,-1)}\put(1,4){\circle*{1}}\end{picture}}}
\newcommand{\btltolandr}{{\begin{picture}(4,5)\put(1,2){\circle*{0.7}}\put(2,3){\circle*{0.7}}\put(3,2){\circle*{0.7}}\put(2,3){\Line(-1,-1)}\put(2,3){\Line(1,-1)}\put(4,4){\circle*{0.7}}\put(4,4){\Line(-2,-1)}\put(4,4){\circle*{1}}\end{picture}}}
\newcommand{\btlandrl}{{\begin{picture}(4,4)\put(1,2){\circle*{0.7}}\put(2,3){\circle*{0.7}}\put(3,1){\circle*{0.7}}\put(4,2){\circle*{0.7}}\put(4,2){\Line(-1,-1)}\put(2,3){\Line(-1,-1)}\put(2,3){\Line(2,-1)}\put(2,3){\circle*{1}}\end{picture}}}
\newcommand{\btrtolandr}{{\begin{picture}(4,5)\put(1,4){\circle*{0.7}}\put(2,2){\circle*{0.7}}\put(3,3){\circle*{0.7}}\put(4,2){\circle*{0.7}}\put(3,3){\Line(-1,-1)}\put(3,3){\Line(1,-1)}\put(1,4){\Line(2,-1)}\put(1,4){\circle*{1}}\end{picture}}}
\newcommand{\btllandr}{{\begin{picture}(4,4)\put(1,1){\circle*{0.7}}\put(2,2){\circle*{0.7}}\put(2,2){\Line(-1,-1)}\put(3,3){\circle*{0.7}}\put(4,2){\circle*{0.7}}\put(3,3){\Line(-1,-1)}\put(3,3){\Line(1,-1)}\put(3,3){\circle*{1}}\end{picture}}}
\newcommand{\btlrr}{{\begin{picture}(4,5)\put(1,3){\circle*{0.7}}\put(2,2){\circle*{0.7}}\put(3,1){\circle*{0.7}}\put(2,2){\Line(1,-1)}\put(1,3){\Line(1,-1)}\put(4,4){\circle*{0.7}}\put(4,4){\Line(-3,-1)}\put(4,4){\circle*{1}}\end{picture}}}
\newcommand{\btlrandr}{{\begin{picture}(4,4)\put(1,2){\circle*{0.7}}\put(2,1){\circle*{0.7}}\put(1,2){\Line(1,-1)}\put(3,3){\circle*{0.7}}\put(4,2){\circle*{0.7}}\put(3,3){\Line(-2,-1)}\put(3,3){\Line(1,-1)}\put(3,3){\circle*{1}}\end{picture}}}

\newcommand{\btlrandrl}{ {\begin{picture}(5,4) \put(1,2){\circle*{0.7}}
\put(2,1){\circle*{0.7}} \put(1,2){\Line(1,-1)}
\put(3,3){\circle*{0.7}} \put(4,1){\circle*{0.7}}
\put(5,2){\circle*{0.7}} \put(5,2){\Line(-1,-1)}
\put(3,3){\Line(-2,-1)} \put(3,3){\Line(2,-1)}
\put(3,3){\circle*{1}} \end{picture}}}

\newcommand{\btlrandrtolandr}{ {\begin{picture}(6,5) \put(1,3){\circle*{0.7}}
\put(2,2){\circle*{0.7}} \put(1,3){\Line(1,-1)}
\put(3,4){\circle*{0.7}} \put(4,2){\circle*{0.7}}
\put(5,3){\circle*{0.7}} \put(6,2){\circle*{0.7}}
\put(5,3){\Line(-1,-1)} \put(5,3){\Line(1,-1)}
\put(3,4){\Line(-2,-1)} \put(3,4){\Line(2,-1)}
\put(3,4){\circle*{1}} \end{picture}}}

\def\smallCompoTwoOneThree{
\begin{psmatrix}[colsep=\colSep,rowsep=\rowSep]
\square & \square\\
& \square\\
& \square & \square & \square
\end{psmatrix} }

\def\smallCompoTwoThree{
\begin{psmatrix}[colsep=\colSep,rowsep=\rowSep]
\square & \square\\
& \square & \square & \square
\end{psmatrix} }

\def\smallCompoTwoOneTwo{
\begin{psmatrix}[colsep=\colSep,rowsep=\rowSep]
\square & \square\\
& \square\\
& \square & \square
\end{psmatrix} }

\def\smallCompoTwoTwo{
\begin{psmatrix}[colsep=\colSep,rowsep=\rowSep]
\square & \square\\
& \square & \square
\end{psmatrix} }

\def\smallCompoTwoOneOne{
\begin{psmatrix}[colsep=\colSep,rowsep=\rowSep]
\square & \square\\
& \square\\
& \square
\end{psmatrix} }

\def\smallCompoTwoOne{
\begin{psmatrix}[colsep=\colSep,rowsep=\rowSep]
\square & \square\\
& \square
\end{psmatrix} }

\def\smallCompoOneTwo{
\begin{psmatrix}[colsep=\colSep,rowsep=\rowSep]
\square\\
\square & \square
\end{psmatrix} }

\def\smallCompoTwo{
\begin{psmatrix}[colsep=\colSep,rowsep=\rowSep] \square & \square\end{psmatrix} }

\def\smallCompoOneOne{
\begin{psmatrix}[colsep=\colSep,rowsep=\rowSep]
\square\\
\square
\end{psmatrix} }

\def\smallCompoOne{
\begin{psmatrix}[colsep=\colSep,rowsep=\rowSep] \square\end{psmatrix} }

\def\smallEmptyCompo{
\begin{psmatrix}[colsep=\colSep,rowsep=\rowSep] \emptyset\end{psmatrix} }

\def\GrTeXBox#1{\vbox{\vskip\vcadre\hbox{\hskip\hcadre%
      $#1$%
   \hskip\hcadre}\vskip\vcadre}}
\def\arx#1[#2]{\ifcase#1 \relax \or%
  \ar @{-}[#2]  \or%
  \ar @2{-}[#2] \or%
  \ar @{--}[#2] \or%
  \ar @2{.}[#2] \or%
  \ar @{~}[#2]  \fi}

    \makeatletter
    \def\MyTextBox{\pst@object{MyTextBox}}
    \def\MyTextBox@i#1#2#3{{%
    \use@par                        
    \if@star\solid@star\fi          
    \pssetlength{\pst@dimc}{-\psframesep}
    \psaddtolength{\pst@dimc}{-\psframesep}
    \psaddtolength{\pst@dimc}{-\pslinewidth}
    \psaddtolength{\pst@dimc}{-\pslinewidth}
    \pssetlength{\pst@dima}{#1}
    \psaddtolength{\pst@dima}{\pst@dimc}
    \pssetlength{\pst@dimb}{#2}
    \psaddtolength{\pst@dimb}{\pst@dimc}
    \setbox\z@\vbox to\pst@dimb{\hsize\pst@dima\sloppy\vfil\small{#3}\vfil}%
    \psframebox{\box\z@}%
    }}\makeatother

    \def\Affect{:=}

\def\Tabvrule{\vrule width-0.3pt}       
\def\Tabhrule{\hrule \hrule height-0.4pt} 
\def\Tabstrut{\vrule height2.2ex 
                     depth0.8ex  
                     width0ex    
\relax}

\def\PasCase#1{\omit%
            $\vcenter{\hbox {\vbox to 0.4pt{}}
               \hbox{\makebox[2ex]{\Tabstrut$#1$}}}%
               \Tabvrule$}
\def\PasCasePoint{\PasCase{\cdot}}
\def\DessinCarre#1{%
    \vcenter{\hbox{}\hrule
             \hbox{\vrule\makebox[3ex]{\Tabstrut$#1$}\vrule}\Tabhrule}%
             \Tabvrule}
\def\GenRuban#1{\vcenter{\halign{&$\DessinCarre{##}$\cr#1}}\egroup}

\def\sTabvrule{\vrule width-0.4pt}
\def\sTabhrule{\hrule \hrule height-0.4pt}
\def\sTabstrut{\vrule height1.2ex depth0.6ex width0ex \relax}
\def\sDessinCarre#1{%
    \vcenter{\hbox{}\hrule
             \hbox{\vrule\makebox[2.3ex]%
                  {\sTabstrut$\scriptstyle#1$}\vrule}\sTabhrule}%
             \sTabvrule}
\def\sGenRuban#1{\vcenter{\halign{&$\sDessinCarre{##}$\cr#1}}\egroup}

\def\ruban{%
  \bgroup
  \let\ =\omit
  \let\\=\cr
  \let\.=\PasCasePoint
  \offinterlineskip
  \GenRuban}
\begin{document}

\title[Fomin's approach for RSK applied to the BST and the Hypoplactic insertions]
{Binary Search Tree insertion, the Hypoplactic insertion,\\and Dual
Graded Graphs}
\author[J. Nzeutchap]{Janvier Nzeutchap}

\address{LITIS EA 4051 (Laboratoire d'Informatique, de Traitement de l'Information et des Syst\`emes)\\
Avenue de l'Universit\'e, 76800 Saint Etienne du Rouvray, France}
\email{janvier.nzeutchap@univ-mlv.fr}
\urladdr{http://monge.univ-mlv.fr$/^{\sim}$nzeutcha}



\begin{abstract}
Fomin (1994) introduced a notion of duality between two graded
graphs on the same set of vertices. He also introduced a
generalization to dual graded graphs of the classical
Robinson-Schensted-Knuth algorithm. We show how Fomin's approach
applies to the binary search tree insertion algorithm also known as
sylvester insertion, and to the hypoplactic insertion algorithm.
\end{abstract}

\maketitle

\vspace{-1cm}

\tableofcontents

\vspace{-1.5cm}

\section{Introduction and definitions}

The Young lattice is defined on the set of partitions of positive
integers, with covering relations given by the natural inclusion
order. The differential poset nature of this graph was generalized
by Fomin with the introduction of graph duality
\cite{graph_duality}. With this extension he introduced
\cite{schensted_for_dual_graphs} a generalization of the classical
Robinson-Schensted-Knuth \cite{schensted_subsequence,
knuth_gen_young_tableaux} algorithm, giving a general scheme for
establishing bijective correspondences between pairs of saturated
chains in dual graded graphs, both starting at a vertex of rank 0
and having a common end point of rank $n$, on the one hand, and
permutations of the symmetric group $\Sn$ on the other hand. This
approach naturally leads to the Robinson-Schensted insertion
algorithm.

Roby \cite{t_roby_thesis} gave an insertion algorithm, analogous to
the Schencted correspondence, for mapping a permutation to a pair of
Young-Fibonacci tableaux, interpreted as saturated chains in the
Fibonacci lattice $Z(1)$ introduced by Stanley
\cite{standley_differential} and also by Fomin \cite{rskGen}, and he
showed that Fomin's approach is partially equivalent to his
construction. He also \cite{roby_oscillating} made a connection
between graded graphs and the Robinson-Schensted correspondence for
skew oscillating tableaux. More recently, Cameron and Killpatrick
\cite{domino_fibo} gave an insertion algorithm, analogous to the
Schensted correspondence, for mapping a colored permutation to a
pair of domino Fibonacci tableaux, interpreted as saturated chains
in the Fibonacci lattice $Z(2)$, and they also showed that Fomin's
approach is partially equivalent to their construction. For both
constructions in $Z(1)$ and $Z(2)$, an evacuation is needed to make
the insertion algorithm coincide with Fomin's approach
\cite{evacuation_killpatrick, domino_fibo}.

The motivation of this note is to do the same for two other
combinatorial insertion algorithms, namely the binary search tree
insertion algorithm of Knuth \cite{knuth_comp_programming} also
known as the sylvester insertion as defined by Hivert et al
\cite{pbt}, and the hypoplactic insertion algorithm of Krob and
Thibon \cite{ncsf4}. In section \ref{section::duality} we recall the
necessary background and definitions on graph duality ; the reader
should refer to \cite{graph_duality} for more details on the
subject. In section \ref{section::hypo} we first recall the
hypoplactic insertion algorithm, then we build isomorphic images of
two dual graded graphs introduced by Fomin, and define a growth
function or $1$-correspondence into those graphs. We end the section
with the application of Fomin's approach using that correspondence
into the two graphs, and we relate the pairs of tableaux obtained
from the hypoplactic insertion algorithm to the pairs of tableaux
obtained from Fomin's growth diagrams. In section
\ref{section::sylv} we perform a similar process for the binary
search tree insertion (BST). The results presented in this note were
first announced in \cite{nzeutchap_fpsac06}.

\subsection{Definitions}\label{section::duality}

\begin{definition}[graded graph]A graded graph is a triple $G = (P, \rho, E)$ where
$P$ is a discrete set of vertices, $\rho : P \to \mathbb Z$ is a
rank function and $E$ is a multi-set of edges $(x, y)$ satisfying
$\rho\,(y) = \rho\,(x) + 1$.\end{definition}%
Let $G = (G_1, G_2) = (P, \rho, E_1, E_2)$ be a pair of graded
graphs with a common set of vertices and a common rank function,
where $E_1$ is the set of $G_1$-edges, directed upwards (in the
direction of increasing rank) and $E_2$ the set of $G_2$-edges,
directed downwards (in the direction of decreasing rank).\\

Let $\mathbb K$ be a field of characteristic zero, define $\mathbb
KP$ as the vector space formed by linear combinations of vertices of
$P$. One can now define two linear operators $U$ (Up) and $D$ (Down)
acting on $\mathbb KP$ as follows. \begin{equation} U\,x \,\, = \,\,
\sum_{(x, y)\,\, \in \,\, E_1} w_1(x, y) \,\, y \qquad ; \qquad D\,y
\,\, = \,\, \sum_{(x,y)\,\, \in \,\, E_2} w_2(x, y) \,\, x
\end{equation} \noindent where $w_i(x, y)$ is the multiplicity or
the weight of the edge $(x, y)$ in $E_i$.

\begin{definition}[graph duality]$G_1$
and $G_2$ are said to be dual \cite{graph_duality} if the two
operators $U$ and $D$ satisfy the commutation relation below.
\begin{equation}\label{equa::duality}D_{n+1}\,U_n \,\, = \,\, U_{n-1}\,D_n +
I_n\end{equation} where $U_n$ (resp. $D_n$) denotes the restriction
of the operator $U$ (resp. $D$) to the $n^{th}$ level of the graph,
and $I_n$ the identical operator at the same level. There are
generalizations of this definition, notably the case of an
$r$-duality with $r > 1$ and the case of an $r_n$-duality, with the
relations bellow.
\begin{equation}\label{equa::dualityGen}D_{n+1}\,U_n \,\, = \,\, U_{n-1}\,D_n +
r\,I_n \quad \mbox{and} \quad D_{n+1}\,U_n \,\, = \,\, U_{n-1}\,D_n
+ r_n\,I_n\end{equation}\end{definition}

A well-known example is the Young lattice of partitions of integers,
which is a \emph{self-dual} graded graph ($G_1 = G_2$) or
\emph{differential poset} \cite{standley_differential}. Its
self-duality expresses the fact that for any Ferrers diagram
$\lambda$, there is one more Ferrers diagram obtained by adding a
single box to $\lambda$ than by deleting a single box from
$\lambda$, and for any couple of Ferrers diagram $(\lambda, \mu)$
there are as many Ferrers diagram simultaneously contained by
$\lambda$ and $\mu$ than those simultaneously containing $\lambda$
and $\mu$.

\begin{definition}[$1$-correspondence]Introduced by Fomin \cite{graph_duality}, it
denotes any bijective map $\phi$ in a self-dual graded graph $G =
(P, \rho, E)$, sending any pair $(b_1,b_2)$ of edges having a common
end point, to a triple $(a_1, a_2, \alpha)$ where $a_1$ and $b_1$
are two edges of the lattice having a common start point, $\alpha$
is either $0$ or $1$, and the following properties are satisfied,
where $a_1$ is the edge $(t,x)$, $a_2 \egaldef (t,y)$, $b_1 \egaldef
(y,z)$ and $b_2 \egaldef (x,z)$.
\begin{enumerate}
  \item $end(a_1) = start(b_2)$ and $end(a_2) = start(b_1)$ ;
  \item if $b_1$ and $b_2$ are degenerated,
  that is to say there exists a vertex $x_0 \in P$ such that $b_1 = b_2 = (x_0,x_0)$,
  then $(a_1, a_2, \alpha) = (b_1,b_2,0)$.
\end{enumerate} \vspace{-0.25cm}
{\small
$$\begin{psmatrix}[colsep=0.6,rowsep=0.25]
 && _{b_2} && \\[0pt]
& [name=x]x && [name=z]\fbox{z} \\[0pt]
^{a_1} && \alpha && _{b_1} \\[0pt]
& [name=t]t && [name=y]y \\[0pt]
 && ^{a_2} &&
  \psset{nodesep=5pt,arrows=->}
  \ncline{t}{x} \ncline{x}{z}
  \ncline{t}{y} \ncline{y}{z}
\end{psmatrix}$$}
\vspace{-1cm}
\begin{figure}[h]
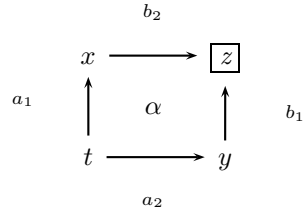
\label{label::square::growth::diagram}
\caption{A square in a growth diagram.}
\end{figure}
\end{definition}

A $1$-correspondence can be used to build a \emph{growth diagram},
which is a sort of pictorial representation of a family of
correspondences similar to the Schensted correspondence. This
approach is introduced by Fomin \cite{schensted_for_dual_graphs} and
concerns a family of insertion algorithms sending permutations onto
pairs of labeled combinatorial objects \emph{of the same shape},
interpreted as saturated chains in the differential poset
considered.

For any permutation $\sigma$, the growth diagram $d(\sigma)$ is
build the following way. First draw the permutation matrix of
$\sigma$ ; next fill the left and lower boundary of $d(\sigma)$ with
the empty combinatorial object generally denoted by $\emptyset$. The
rest of the construction is iterative ; $d(\sigma)$ is filled from
its lower left corner to its upper right corner, following the
diagonal. At each step and for any configuration $(a_1,a_2,\alpha)$
as pictured above (Fig. 1), 
$z$ is obtained by application of the $1$-correspondence to $t$,
$x$, $y$ and $\alpha$.

\section[Fomin's approach applied to the hypoplactic insertion algorithm]
{Fomin's approach applied to the hypoplactic insertion
algorithm}\label{section::hypo}


\def\mya{{\bf a}}

Introduced in \cite{ncsf4}, the hypoplactic correspondence is an
insertion algorithm, analogous to the Robinson-Schensted
correspondence, mapping a permutation to a pair made of a
quasi-ribbon tableau and a ribbon tableau. It appears in the study
of noncommutative symmetric functions. A ribbon tableau is a
composition diagram filled with positive integers in such a way that
entries increase across lines from left to right, and up columns. In
a quasi-ribbon diagram, entries increase down columns. As usual, the
insertion tableau of a permutation is iteratively constructed
reading its letters from left to right. To insert a letter $\mya$ in
a quasi-ribbon, compare $\mya$ with the last letter $z$ in its last
row. If $\mya$ is greater then just append it to the right of $z$.
Otherwise, reading the quasi-ribbon from left to right and top to
bottom, find the last entry $y$ such that $y \leq \mya$, then insert
a cell labeled $\mya$ just to the right of the one labeled $y$ and
shift the rest of the quasi-ribbon below the newly created cell. Let
us apply the algorithm to the permutation $\sigma = 415362$.

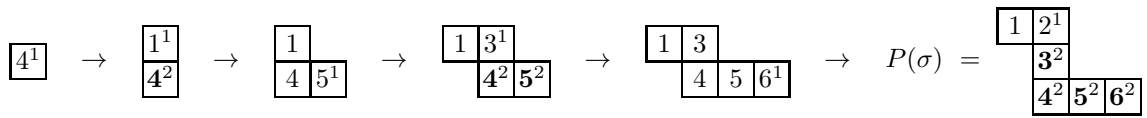
\def\myspace{\quad}

\begin{example}\label{example::hypo}
{\small $$ \ruban{4^1\\} \myspace \to \myspace \ruban{1^1\\{\bf 4}^2\\} \myspace \to \myspace \ruban{1\\4&5^1\\}%
\myspace \to \myspace \ruban{1&3^1\\\ &{\bf 4}^2&{\bf 5}^2\\}
\myspace \to \myspace \ruban{1&3\\\ &4&5&6^1\\} \myspace \to
\myspace P(\sigma) \,\, = \,\, \ruban{1&2^1\\ \ &{\bf 3}^2\\ \ &{\bf
4}^2&{\bf 5}^2&{\bf 6}^2\\}$$}
\end{example}

\noindent where $x^1$ means that inserting or appending the cell
labeled $x$ is the first action performed during the current step,
and the cells with a bold entry are the ones shifted down, this is
the second action of the current step.

Note that the shape of the quasi-ribbon $P(\sigma)$ is the recoils
composition of $\sigma$ (that is the descents composition of
$\sigma^{-1}$), and that $P(\sigma)$ is canonically labeled from
left to right and from top to bottom. The labels of the cells of the
ribbon tableau $Q(\sigma)$ record the positions in $\sigma$ of the
labels of the cells of $P(\sigma)$. {\small
$$ \ruban{1\\} \myspace \to \myspace \ruban{2\\1\\} \myspace \to \myspace \ruban{2\\1&3\\}%
\myspace \to \myspace \ruban{2&4\\\ &1&3\\} \myspace \to \myspace
\ruban{2&4\\\ &1&3&5\\} \myspace \to \myspace Q(\sigma) \,\, = \,\,
\ruban{2&6\\ \ &4\\ \ &1&3&5\\}$$}

Now let us introduce the two dual graded graphs we will use to make
the connection between the hypoplactic insertion algorithm and
Fomin's approach for RSK.

\subsection{Dual graded graphs on compositions of integers}\label{section::hypo::graphs}\

\def\liftedbintree{\emph{lifted binary tree}}
\def\reflectedliftedbintree{\emph{reflected lifted binary tree}}
\def\binword{\emph{Binword}}


In this section, we are interested in isomorphic images of two
graphs studied by Fomin, namely the \liftedbintree\ and \binword\
 (\cite{graph_duality}, Example 2.4.1 and Fig.12). Their vertices are
words on the alphabet $\{\,0, 1\}$.
\begin{enumerate}
  \item in the \liftedbintree, a word $w$ is covered by the two words $w.0$
  and $w.1$ (where . denotes the usual concatenation of words), except
  $0$ is only covered by $1$ ;
  \item in \binword, there is an edge $(u, v)$ if $u$ is obtained by deleting a single letter (but
  not the first one) from $v$ ; in addition there is an edge $(0, 1)$.
\end{enumerate}

\begin{lemma}\label{lemma::compo::to::words}
There is a one-to-one correspondence between compositions of an
integer $n$ and $n$-letter words in the alphabet $\{\, 0,1\}$,
mapping a composition $c$ to the word read after filling its diagram
from left to right and top to bottom, with 1's in the first box and
in any box following a descent position. {\small
$$ c \,\, = \,\, 321 \quad \equiv \quad \ruban{ & &\\ \ & \ &  & \\ \ & \ & \ &
\\} \qquad ; \qquad \ruban{1 & 0  & 0\\ \ & \ & 1 &
0 \\ \ & \ & \ & 1\\} \quad \equiv \quad w_c \quad = \quad
100101$$}\end{lemma}

With this lemma, we can now build two dual graded graphs whose
vertices of rank $n$ are all the compositions of the integer $n$,
with covering relations obtained by expressing the ones above on
compositions of integers rather than on $\{\, 0,1\}$-words. So in
the first graph, a composition $c \neq \emptyset$ is covered by two
compositions obtained either by increasing its last part, or by
appending $1$ after its last part ; in addition there is an edge
$(\emptyset, 1)$. In the second graph, a composition $c \neq
\emptyset$ is covered by the compositions obtained either by
increasing a single part or by inserting a single 1, or by first
splitting a single part into two parts and then increase one of the
parts obtained, in addition there is an edge $(\emptyset, 1)$. This
version of \binword\ appeared in \cite{anaYoungComp}.

\begin{multicols}{2}

\def\colSep{0.1}
\def\rowSep{0.9}
{\small
$$\begin{psmatrix}[colsep=\colSep,rowsep=\rowSep]
[name=c4]4 &&  [name=c31]31 && [name=c22]22 && [name=c211]211
&[name=c13]13 && [name=c121]121 && [name=c112]112 && [name=c1111]1111\\[0pt]
& [name=c3]3 &&&&  [name=c21]21 &&& [name=c12]12 &&&& [name=c111]111\\[0pt]
&&&[name=c2]2 &&&&&&&&[name=c11]11\\[0pt]
&&&&&&&[name=c1]1\\[0pt]
&&&&&&&[name=c0]\emptyset
\psset{nodesep=3pt,arrows=-} \ncline{c0}{c1} \ncline{c1}{c2}
\ncline{c1}{c11} \ncline{c2}{c3} \ncline{c2}{c21} \ncline{c11}{c12}
\ncline{c11}{c111}%
\ncline{c3}{c4} \ncline{c3}{c31}%
\ncline{c21}{c22} \ncline{c21}{c211}%
\ncline{c12}{c13} \ncline{c12}{c121}%
\ncline{c111}{c112} \ncline{c111}{c1111}%
\end{psmatrix}$$}

\columnbreak

\def\colSep{0.25}
\def\rowSep{0.9}
{\small
$$\begin{psmatrix}[colsep=\colSep,rowsep=\rowSep]
[name=c4]4 &&  [name=c31]31 && [name=c22]22 && [name=c211]211
&&[name=c13]13 && [name=c121]121 && [name=c112]112 && [name=c1111]1111\\[0pt]
& [name=c3]3 &&&&  [name=c21]21 &&& [name=c12]12 &&&& [name=c111]111\\[0pt]
&&&[name=c2]2 &&&&&&&&[name=c11]11\\[0pt]
&&&&&&&[name=c1]1\\[0pt]
&&&&&&&[name=c0]\emptyset
\psset{nodesep=3pt,arrows=-} \ncline{c0}{c1} \ncline{c1}{c2}
\ncline{c1}{c11}%
\ncline{c2}{c3} \ncline{c2}{c21} \ncline{c2}{c12}%
\ncline{c11}{c12} \ncline{c11}{c111} \ncline{c11}{c21}%
\ncline{c3}{c4} \ncline{c3}{c31} \ncline{c3}{c13} \ncline{c3}{c22}%
\ncline{c21}{c31} \ncline{c21}{c22} \ncline{c21}{c211} \ncline{c21}{c121}%
\ncline{c12}{c22} \ncline{c12}{c13} \ncline{c12}{c121} \ncline{c12}{c112}%
\ncline{c111}{c211} \ncline{c111}{c121} \ncline{c111}{c112} \ncline{c111}{c1111}%
\end{psmatrix}$$}

\end{multicols}

\vspace{-0.75cm}
\begin{figure}[h]\caption{The \emph{\liftedbintree}\ and \binword\ graphs on compositions of integers.}\end{figure}

\begin{proposition}\label{proposition::rcor::hypo}The following describes a $1$-correspondence
in the \emph{\liftedbintree}\ and \emph{\binword}\ as defined above
on compositions on integers.
\end{proposition}

\begin{algorithm}
\caption{A natural $r$-correspondence in the \liftedbintree\ and
\binword.} \label{algo::rcor::qsym::ncsf}
\begin{algorithmic}[1]
\IF{$t = x = y$ and $\alpha = 1$}%
    \STATE $z \Affect x$, with its last part increased%
\ELSE %
    \IF{$x = y$}%
        \STATE $z \Affect x $, with an additional 1 at the end%
    \ELSE %
        \IF{$x$ ends with 1}%
            \STATE $z \Affect x $, with an additional 1 at the end%
        \ENDIF%
    \ENDIF%
\ENDIF
\end{algorithmic}
\end{algorithm}

\begin{proof}
To show that this description is a $1$-correspondence, one may
consider the one described in
(\cite{schensted_for_dual_graphs}-Lemma 4.6.1) and replace $y0$ by
$y1$ wherever it appears, and conversely. One gets a second
$1$-correspondence $\phi$ in the \emph{\liftedbintree}\ and
\emph{\binword}\ as defined on $\{\,0, 1\}$-words. Now apply Remark
\ref{lemma::compo::to::words} to $\phi$ and you get the description
above.\end{proof}

In the next section, we show that using this $1$-correspondence to
build Fomin's growth diagram for any permutation $\sigma$, one gets
two saturated chains that can be translated into a quasi-ribbon
tableau $\fominPhypo(\sigma)$ and a ribbon tableau
$\fominQhypo(\sigma)$, and one naturally has $P(\sigma) =
\fominPhypo(\sigma)$ and $Q(\sigma) = \fominQhypo(\sigma)$.

\def\growthdiagramOfSigma{d(\sigma)}

\subsection{Growth diagram - equivalence of the two
constructions}\label{section::hypo::fomin}\

Let us build the growth diagram $\growthdiagramOfSigma$ of the
permutation $\sigma = 415362$, using the $1$-correspondence defined
in Proposition \ref{proposition::rcor::hypo}. On the upper boundary
of $\growthdiagramOfSigma$, one gets a satureted chain $\fominQhypo$
in the \liftedbintree, and on the right boundary a chain
$\fominPhypo$ in \binword. Fomin \cite{finite_posets} gave a
geometric construction for the Schensted insertion algorithm,
Cameron and Killpatrick \cite{evacuation_killpatrick, domino_fibo}
did the same for the Young-Fibonacci and the domino-Fibonacci
insertion algorithms. We will now do the same for the hypoplactic
insertion algorithm. It will consist in drawing shadow lines that
can be used to directly determine the tableaux $P(\sigma)$ and
$Q(\sigma)$ obtained through the hypoplactic insertion algorithm. To
begin with, draw the permutation matrix of $\sigma$. Then from
bottom to top, draw a set of broken lines to join the symbol
$\permOneSymb$ on the current line of the permutation matrix to the
one on the line just above if the second symbol $\permOneSymb$ is
situated to the right of the first one, if this is not the case then
start a new broken line with the second symbol $\permOneSymb$.
$P(\sigma)$ corresponds to reading vertical coordinates of the
symbols $\permOneSymb$ on the broken lines, from left to right and
from bottom to top. As for $Q(\sigma)$, it corresponds to reading
horizontal coordinates of the symbols $\permOneSymb$ on the broken
lines, in the same order.

\def\colSep{-0.04}
\def\rowSep{-0.2}
\def\corner{\square}%

{\small \centerline{
\newdimen\vcadre\vcadre=0.15cm 
\newdimen\hcadre\hcadre=0.15cm 
\setlength\unitlength{1.2mm} $\xymatrix@R=0cm@C=0mm{
*{\GrTeXBox{\smallEmptyCompo}}\arx1[dd]\arx1[rr] && *{\GrTeXBox{\smallCompoOne}}\arx1[dd]\arx1[rr] && *{\GrTeXBox{\smallCompoOneOne}}\arx1[dd]\arx1[rr] &&%
*{\GrTeXBox{\smallCompoOneTwo}}\arx1[dd]\arx1[rr] && *{\GrTeXBox{\smallCompoTwoTwo}}\arx1[dd]\arx1[rr] && *{\GrTeXBox{\smallCompoTwoThree}}\arx1[dd]\arx1[rr] &&%
*{\GrTeXBox{\smallCompoTwoOneThree}}\arx1[dd]\\
& *{\GrTeXBox{}} && *{\GrTeXBox{}} && *{\GrTeXBox{}} &&%
*{\GrTeXBox{}} && *{\GrTeXBox{\permOneSymb}}\arx3[dd] && *{\GrTeXBox{}}  && *{\GrTeXBox{6}}\\
*{\GrTeXBox{\smallEmptyCompo}}\arx1[dd]\arx1[rr] && *{\GrTeXBox{\smallCompoOne}}\arx1[dd]\arx1[rr] && *{\GrTeXBox{\smallCompoOneOne}}\arx1[dd]\arx1[rr] &&%
*{\GrTeXBox{\smallCompoOneTwo}}\arx1[dd]\arx1[rr] && *{\GrTeXBox{\smallCompoTwoTwo}}\arx1[dd]\arx1[rr] && *{\GrTeXBox{\smallCompoTwoTwo}}\arx1[dd]\arx1[rr] &&%
*{\GrTeXBox{\smallCompoTwoOneTwo}}\arx1[dd]\\
& *{\GrTeXBox{}} && *{\GrTeXBox{}} && *{\GrTeXBox{\permOneSymb}}\arx3[dd]\arx3[rrrr] &&%
*{\GrTeXBox{}} && *{\GrTeXBox{}} && *{\GrTeXBox{}} && *{\GrTeXBox{5}}\\
*{\GrTeXBox{\smallEmptyCompo}}\arx1[dd]\arx1[rr] && *{\GrTeXBox{\smallCompoOne}}\arx1[dd]\arx1[rr] && *{\GrTeXBox{\smallCompoOneOne}}\arx1[dd]\arx1[rr] &&%
*{\GrTeXBox{\smallCompoOneOne}}\arx1[dd]\arx1[rr] && *{\GrTeXBox{\smallCompoTwoOne}}\arx1[dd]\arx1[rr] && *{\GrTeXBox{\smallCompoTwoOne}}\arx1[dd]\arx1[rr] &&%
*{\GrTeXBox{\smallCompoTwoOneOne}}\arx1[dd]\\
& *{\GrTeXBox{\permOneSymb}}\arx3[rrrr] && *{\GrTeXBox{}} && *{\GrTeXBox{}} &&%
*{\GrTeXBox{}} && *{\GrTeXBox{}} && *{\GrTeXBox{}} && *{\GrTeXBox{4}}\\
*{\GrTeXBox{\smallEmptyCompo}}\arx1[dd]\arx1[rr] && *{\GrTeXBox{\smallEmptyCompo}}\arx1[dd]\arx1[rr] && *{\GrTeXBox{\smallCompoOne}}\arx1[dd]\arx1[rr] &&%
*{\GrTeXBox{\smallCompoOne}}\arx1[dd]\arx1[rr] && *{\GrTeXBox{\smallCompoTwo}}\arx1[dd]\arx1[rr] && *{\GrTeXBox{\smallCompoTwo}}\arx1[dd]\arx1[rr] &&%
*{\GrTeXBox{\smallCompoTwoOne}}\arx1[dd]\\
& *{\GrTeXBox{}} && *{\GrTeXBox{}} && *{\GrTeXBox{}} &&%
*{\GrTeXBox{\permOneSymb}} && *{\GrTeXBox{}} && *{\GrTeXBox{}} && *{\GrTeXBox{3}}\\
*{\GrTeXBox{\smallEmptyCompo}}\arx1[dd]\arx1[rr] && *{\GrTeXBox{\smallEmptyCompo}}\arx1[dd]\arx1[rr] && *{\GrTeXBox{\smallCompoOne}}\arx1[dd]\arx1[rr] &&%
*{\GrTeXBox{\smallCompoOne}}\arx1[dd]\arx1[rr] && *{\GrTeXBox{\smallCompoOne}}\arx1[dd]\arx1[rr] && *{\GrTeXBox{\smallCompoOne}}\arx1[dd]\arx1[rr] &&%
*{\GrTeXBox{\smallCompoTwo}}\arx1[dd]\\
& *{\GrTeXBox{}} && *{\GrTeXBox{}} && *{\GrTeXBox{}} &&%
*{\GrTeXBox{}} && *{\GrTeXBox{}} && *{\GrTeXBox{\permOneSymb}}\arx3[dd] && *{\GrTeXBox{2}}\\
*{\GrTeXBox{\smallEmptyCompo}}\arx1[dd]\arx1[rr] && *{\GrTeXBox{\smallEmptyCompo}}\arx1[dd]\arx1[rr] && *{\GrTeXBox{\smallCompoOne}}\arx1[dd]\arx1[rr] &&%
*{\GrTeXBox{\smallCompoOne}}\arx1[dd]\arx1[rr] && *{\GrTeXBox{\smallCompoOne}}\arx1[dd]\arx1[rr] && *{\GrTeXBox{\smallCompoOne}}\arx1[dd]\arx1[rr] &&%
*{\GrTeXBox{\smallCompoOne}}\arx1[dd]\\
& *{\GrTeXBox{}} && *{\GrTeXBox{\permOneSymb}}\arx3[rrrrrrrr] && *{\GrTeXBox{}} &&%
*{\GrTeXBox{}} && *{\GrTeXBox{}} && *{\GrTeXBox{}}  && *{\GrTeXBox{1}}\\
*{\GrTeXBox{\smallEmptyCompo}}\arx1[rr] && *{\GrTeXBox{\smallEmptyCompo}}\arx1[rr] && *{\GrTeXBox{\smallEmptyCompo}}\arx1[rr] &&%
*{\GrTeXBox{\smallEmptyCompo}}\arx1[rr] && *{\GrTeXBox{\smallEmptyCompo}}\arx1[rr] && *{\GrTeXBox{\smallEmptyCompo}}\arx1[rr] &&%
*{\GrTeXBox{\smallEmptyCompo}}\\
& *{\GrTeXBox{1}} && *{\GrTeXBox{2}} && *{\GrTeXBox{3}}  && *{\GrTeXBox{4}} && *{\GrTeXBox{5}} && *{\GrTeXBox{6}}\\
 }$ } } \vspace{-0.5cm}
\begin{figure}[h]
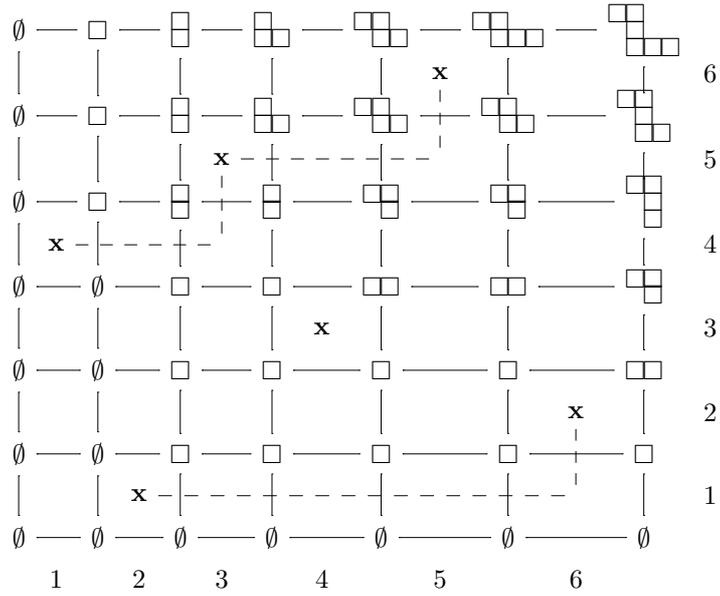
\label{figure::growth::diagram::hypo}
\caption{Example of growth diagram for the hypoplactic insertion
algorithm.}
\end{figure}

Now one question which naturally presents itself is how to convert
the chain $\fominPhypo$ into a quasi-ribbon tableau
$\fominPhypo(\sigma)$ and the chain $\fominQhypo$ into a ribbon
tableau  $\fominQhypo(\sigma)$.\\

\subsubsection{From a chain of compositions to a standard
quasi-ribbon tableau}\

\def\mysymbol{{\bf \centerdot}}
Recall that in the chain $\fominPhypo$, a composition is covered by
another one obtained either by increasing its last part, or by
appending $1$ after this last part, in addition $\emptyset$ is
covered by $1$.

{\small $$ \fominPhypo \myspace = \myspace \emptyset \myspace \to
\myspace {\bf 1} \myspace \to \myspace 2\mysymbol \myspace \to
\myspace 21\mysymbol \myspace \to \myspace 21{\bf 1} \myspace \to
\myspace 21{\bf 2} \myspace \to \myspace 213$$}

So in order to get a quasi-ribbon tableau, one is simply to label
the cells appearing in the right boundary of $\growthdiagramOfSigma$
in the order they occur, and it is clear that this will always
produce a quasi-ribbon tableau.

{\small $$\emptyset \myspace \to \myspace \ruban{1^1\\} \myspace \to
\myspace \ruban{1&2^1\\}
\myspace \to \myspace \ruban{1&2\\ \ &3^1\\} \myspace \to \myspace \ruban{1&2\\ \ &3\\
\ &4^1\\} \myspace \to \myspace \ruban{1&2\\ \
&3\\ \ &4&5^1\\} \myspace \to \myspace \fominPhypo(\sigma) \myspace = \myspace \ruban{1&2\\ \ &3\\
\ &4&5&6^1\\}$$}

\subsubsection{From a chain of compositions to a standard
ribbon tableau}\

Recall that in the chain $\fominQhypo$, a composition is covered by
another one obtained either by increasing a single part, or by
appending $1$ after any part (eventually just in front), in addition
there is an edge $(\emptyset, 1)$.

{\small $$ Q \myspace = \myspace \emptyset \myspace \to \myspace
1\mysymbol \myspace \to \myspace 1{\bf 1} \myspace \to \myspace {\bf
1}2 \myspace \to \myspace 2{\bf 2} \myspace \to \myspace 2\mysymbol3
\myspace \to \myspace 213$$}

In order to get a ribbon tableau, we will process $\fominQhypo$ as
follows.
\begin{enumerate}
  \item when one part has been increased at step $k$,
  we append a cell labeled $k$ just at the end of the corresponding
  line, and the lower part of the ribbon tableau is shifted to the right ;
  \item when $1$ has been inserted after the $i^{th}$ part,
  we shift down the portion of the ribbon tableau starting at the last cell on the $i^{th}$ line,
  and then we insert a cell labeled $k$ at this position.
\end{enumerate}

It is clear that this will always produce a ribbon tableau.

{\small $$ \emptyset \myspace \to \myspace \ruban{1^1\\} \myspace
\to \myspace \ruban{2^2\\{\bf 1}^1\\} \myspace \to \myspace
\ruban{2& \ \\1&3^1\\} \myspace \to \myspace \ruban{2&4^2\\ \ &{\bf 1}^1&{\bf 3}^1\\} \myspace \to \myspace \ruban{2&4\\
\ &1&3&5^1\\} \myspace \to \myspace Q(\sigma) \myspace = \myspace \ruban{2&6^2\\
\ &{\bf 4}^1\\ \ &{\bf 1}^1&{\bf 3}^1&{\bf 5}^1\\}$$}

\begin{proposition}For any permutation $\sigma$,
$P(\sigma) = \fominPhypo(\sigma)$ and $Q(\sigma) =
\fominQhypo(\sigma)$.
\end{proposition}

\begin{proof}
This follows from the definition of the hypoplactic insertion
algorithm, and the remark that any composition $\fominQhypo_k$
appearing in $\fominQhypo$ is the shape of the tableau $P(\sigma_1
\sigma_2 \cdots \sigma_k)$ where $\sigma_1 \sigma_2 \cdots \sigma_k$
is the restriction of $\sigma$ to its first $k$ letters, and any
composition $\fominPhypo_k$ appearing in $\fominPhypo$ is the shape
of $P(\sigma_{/[1..k]})$ where $\sigma_{/[1..k]}$ is the restriction
of $\sigma$ to the interval $[1..k]$. Indeed, with this remark, the
canonical labelling of $P(\sigma)$ coincides with the description of
the conversion of $\fominPhypo$ into $\fominPhypo(\sigma)$. As for
$\fominQhypo$, its conversion into $Q(\sigma)$ clearly coincides
with the description of the hypoplactic insertion algorithm (see
Example \ref{example::hypo}).
\end{proof}

\section{Fomin's approach applied to the sylvester insertion algorithm}\label{section::sylv}

\subsection{The insertion algorithm}\label{section::sylv::algo}\

The sylvester insertion algorithm is a variant of the BST
\cite{knuth_comp_programming}. Introduced by Hivert et al
\cite{pbt}, it was used to give a new construction in term of
noncommutative polynomials, of the algebra of Planar Binary Trees of
Loday-Ronco \cite{loday_ronco_pbt}. Trough this algorithm, the
insertion tree of a given permutation $\sigma$ is the binary search
tree built by \emph{\underline{reading $\sigma$ from right to
left}}. The recording tree is a \emph{decreasing tree}, that is to
say a labeled binary tree such that the label of each internal node
is greater than the labels of all the nodes in its subtrees. It
records the positions in $\sigma$ of the labels of the insertion
tree.



In the classical version of the BST, permutations are \emph{read
from left to right} and the insertion tree will be denoted
$\knuthbtree(\sigma)$. The recording tree $\knuthitree(\sigma)$ is
an \emph{increasing tree}, that is to say a labeled binary tree such
that the label of each internal node is smaller than the labels of
all the nodes in its subtrees. So for $\sigma = 351426$, one gets
the following insertion and recording trees. {\small
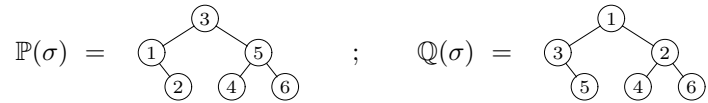
$$ \knuthbtree(\sigma) \,\, = \,\,
\vcenter{\tiny\xymatrix@C=0mm@R=1mm{
      &    &&  *=<10pt>[o][F-]{3}\ar@{-}[dll]\ar@{-}[drr] \\
      & *=<10pt>[o][F-]{1}\ar@{-}[dr] && && *=<10pt>[o][F-]{5}\ar@{-}[dl]\ar@{-}[dr] \\
      &&  *=<10pt>[o][F-]{2} && *=<10pt>[o][F-]{4} && *=<10pt>[o][F-]{6}
    }}
\qquad ; \qquad \knuthitree(\sigma) \,\, = \,\,
\vcenter{\tiny\xymatrix@C=0mm@R=1mm{
      &    &&  *=<10pt>[o][F-]{1}\ar@{-}[dll]\ar@{-}[drr] \\
      & *=<10pt>[o][F-]{3}\ar@{-}[dr] && && *=<10pt>[o][F-]{2}\ar@{-}[dl]\ar@{-}[dr] \\
      &&  *=<10pt>[o][F-]{5} && *=<10pt>[o][F-]{4} && *=<10pt>[o][F-]{6}
    }} $$}

In section \ref{section::sylv::fomin} we show how Fomin's approach
applies to the BST, but let us first introduce the corresponding
dual graded graphs.

\subsection{Dual graded graphs on binary trees}\label{section::sylv::graphs}\

\def\textpbtqleft{\emph{lattice of binary trees}}
\def\textpbtpleft{\emph{bracket tree}}
\def\textpbtpright{\emph{reflected bracket tree}}

Once again, we will consider isomorphic images of two dual graded
graphs studied by Fomin, namely the \textpbtqleft\ and the
\textpbtpleft\ (\cite{graph_duality}, Fig.13). The \textpbtqleft\ is
defined as follows. Its vertices of rank $n$ are the syntactically
correct formulae defining different versions of calculation of a
\underline{non-associative} product of $n+1$ entries. So any vertex
of rank $n$ is a valid sequence of \ $n\!-\!1$ opening and $n\!-\!1$
closing brackets inserted into $x_1 \, x_2 \cdots x_n$. In the
\textpbtpleft, two expressions are linked if one results from the
other by deleting the first entry, and then removing subsequent
unnecessary brackets, and renumbering the new expression. In the
sequel we will be considering the \textpbtpright, the graph where
two expressions are linked if one results from the other by deleting
the last entry, and then removing subsequent unnecessary brackets,
and renumbering the new expression.

\begin{lemma}\label{lemma::tree::to::bracketed::expressions}
There is a one-to-one correspondence between unlabeled binary trees
and bracketed expressions. A tree $t$ is identified with the
expression obtained by completion of $t$, adding one leaf to any
node having a single child-node, and two leaves to any childless
node. Then label the $n$ leaves of the resulting binary tree $t\,'$
according to the right-to-left infix order, using each of the labels
$1, 2, \cdots, n$ only once. If $t\,'$ is empty then the bracketed
expression is $x_1$, else the bracketed expression is obtained by
recursively reading its right and left subtrees. Below is an example
where the expression obtained is $(x_1 x_2)((x_3
x_4)x_5)$.\end{lemma}

\vspace{-0.5cm}
\def\unnoeud{{\large \bullet}}
\def\petitrayon{5pt}
\def\grandrayon{12pt}

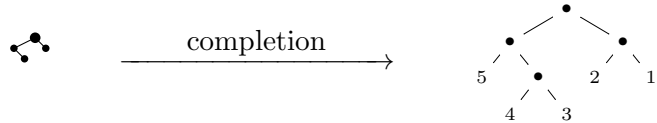
$$
\begin{CD}
    \setlength\unitlength{1.4mm}
  \btlrandr
 \qquad  @>\qquad \mbox{completion}\qquad>> \qquad
  \vcenter{\tiny\xymatrix@C=0mm@R=1mm{
      &    &&  \unnoeud \ar@{-}[dll]\ar@{-}[drr] \\
      & \unnoeud \ar@{-}[dl]\ar@{-}[dr] && && \unnoeud \ar@{-}[dl] \ar@{-}[dr] \\
      5 &&  \unnoeud \ar@{-}[dl] \ar@{-}[dr] && 2 && 1 \\
      & 4 && 3
    }}
\end{CD}$$

\vskip 5pt

With this lemma, we can build two graphs whose vertices of rank $n$
are all the binary trees having $n$ nodes, with covering relations
obtained by expressing the ones above on binary trees rather than on
bracketed expressions. So one will have the following covering
relations.
\begin{enumerate}
  \item in the \textpbtqleft, any tree is covered by all those obtained from it by addition
  of a single node, in all possible ways. Below is a finite realization of the graph, from rank $0$ to rank $4$.

\begin{figure}[h]\label{fig::pbtql}\setlength\unitlength{1.2mm}
\begin{psmatrix}[colsep=0.4,rowsep=0.5]
  [name=btlll]\btlll & [name=btllr]\btllr & [name=btltolandr]\btltolandr & [name=btllandr]\btllandr
  & [name=btlrl]\btlrl & [name=btlrr]\btlrr & [name=btlrandr]\btlrandr
  & [name=btlandrl]\btlandrl & [name=btlandrr]\btlandrr
  & [name=btrll]\btrll & [name=btrlr]\btrlr & [name=btrtolandr]\btrtolandr
  & [name=btrrl]\btrrl & [name=btrrr]\btrrr \\[0pt]
  & [name=btll]\btll &&& [name=btlr]\btlr &&& [name=btlandr]\btlandr
  &&& [name=btrl]\btrl && [name=btrr]\btrr\\[0pt]
  &&&& [name=btl]\btl &&&&&& [name=btr]\btr \\[0pt]
  &&&&&&& [name=bt1]$\btone$\\[0pt]
  &&&&&&& [name=bt0]$\emptyset$
  \psset{nodesep=5pt,arrows=-}
  \ncline{btll}{btlll} \ncline{btll}{btllr} \ncline{btll}{btltolandr} \ncline{btll}{btllandr}%
  \ncline{btlr}{btltolandr} \ncline{btlr}{btllandr}
  \ncline{btlr}{btlrl} \ncline{btlr}{btlrr}%
  \ncline{btlandr}{btllandr} \ncline{btlandr}{btlrandr}
  \ncline{btlandr}{btlandrl} \ncline{btlandr}{btlandrr}
  \ncline{btrl}{btlandrl} \ncline{btrl}{btrll} \ncline{btrl}{btrlr} \ncline{btrl}{btrtolandr}
  \ncline{btrr}{btlandrr} \ncline{btrr}{btrtolandr} \ncline{btrr}{btrrl} \ncline{btrr}{btrrr}
  \ncline{btr}{btlandr} \ncline{btr}{btrl} \ncline{btr}{btrr}
  \ncline{btl}{btll} \ncline{btl}{btlr} \ncline{btl}{btlandr}
  \ncline{bt1}{btl} \ncline{bt1}{btr} \ncline{bt0}{bt1}
\end{psmatrix} \caption{The \emph{\textpbtqleft}.}
\end{figure}


  \item in the \textpbtpright, a tree $y$ covers a single tree $t$
  obtained from $y$ by deleting it's right-most node if any, or its root otherwise, and
  replacing the deleted node by its own left subtree if any.
  Below is a finite realization of the graph, from rank $0$ to rank
  $4$.\\


\def\rowSep{.2cm}
\def\colSep{0.5mm}

{\small \centerline{
\newdimen\vcadre\vcadre=0.2cm 
\newdimen\hcadre\hcadre=0.2cm 
\setlength\unitlength{1.2mm} $\xymatrix@R=\rowSep@C=\colSep{
*{\GrTeXBox{\btlll}}\arx1[dr] && *{\GrTeXBox{\btllandr}}\arx1[dl]%
& *{\GrTeXBox{\btltolandr}}\arx1[dr] & *{\GrTeXBox{\btlandrl}}\arx1[d] & *{\GrTeXBox{\btlandrr}}\arx1[dl] %
& *{\GrTeXBox{\btllr}}\arx1[d] & *{\GrTeXBox{\btlrandr}}\arx1[dl]%
& *{\GrTeXBox{\btlrl}}\arx1[dr] & *{\GrTeXBox{\btrll}}\arx1[d] & *{\GrTeXBox{\btrtolandr}}\arx1[dl] %
& *{\GrTeXBox{\btlrr}}\arx1[drr] & *{\GrTeXBox{\btrlr}}\arx1[dr] & *{\GrTeXBox{\btrrl}}\arx1[d] && *{\GrTeXBox{\btrrr}}\arx1[dll]\\
& *{\GrTeXBox{\btll}}\arx1[drr] &&& *{\GrTeXBox{\btlandr}}\arx1[dl] && *{\GrTeXBox{\btlr}}\arx1[drrrr] &&& *{\GrTeXBox{\btrl}}\arx1[dr] &&&& *{\GrTeXBox{\btrr}}\arx1[dlll]\\
&&& *{\GrTeXBox{\btl}}\arx1[drrrr] &&&&&&& *{\GrTeXBox{\btr}}\arx1[dlll] \\
&&&&&&& *{\GrTeXBox{\btone}}\arx1[d] \\
&&&&&&& *{\GrTeXBox{\emptyset}} }$ }}

\vspace{-0.5cm}

\begin{figure}[ht]\label{fig::pbtpl}\caption{The \emph{\textpbtpright}, a dual of the
\emph{\textpbtqleft}.}\end{figure}

\end{enumerate}

\begin{remark}The reader should pay attention that for a more
convenient graphical representation, vertices at rank $4$ are not
listed in the same order in the two graphs above.
\end{remark}

\begin{proposition}The \emph{\textpbtpright}\ is dual to the
\emph{\textpbtqleft}.
\end{proposition}

\begin{proof}Follows from the duality of the \emph{\textpbtqleft}\ and the \emph{\textpbtpleft}.
\end{proof}\

\begin{proposition}\label{proposition::rcor::sylv::right}The following algorithm describes a $1$-correspondence in
the \textpbtqleft\ and the \textpbtpright.
\end{proposition}

\begin{algorithm}
\caption{natural $r$-correspondence in the \textpbtqleft\ and the
\textpbtpright.} \label{algo::rcor::qsym::ncsf}
\begin{algorithmic}[1]
\IF{$t = x = y$ and $\alpha = 1$}%
    \STATE $z \Affect t $, with one node added as right child-node of its rightmost node%
\ELSE %
    \IF{$x = y$}%
        \STATE $z \Affect y$, with one node added as left child-node of its rightmost node%
    \ELSE %
            \STATE $z \Affect x $, with one node added in such a way that deleting the right-most one gives back $y$ $^{(1)}$%
    \ENDIF%
\ENDIF
\end{algorithmic}
\end{algorithm}
\vspace{-0.35cm} {\footnotesize $^1$ Due to the duality of the two
graphs, there is one and only one way doing this.}

\begin{remark}This algorithm does not apply to degenerated cases for which $z$ is
trivially deduced from the definition of a $1$-correspondence.
\end{remark}

\subsection{Growth diagram - equivalence of the two
constructions}\label{section::sylv::fomin}\

Let us build the growth diagram $\growthdiagramOfSigma$ of the
permutation $\sigma = 351426$. On the upper boundary of
$\growthdiagramOfSigma$, one gets a saturated chain $\fominQSylv$ in
the \textpbtqleft, and on the right boundary a chain $\fominPSylv$
in the \textpbtpright.

{\small \centerline{
\newdimen\vcadre\vcadre=0.15cm 
\newdimen\hcadre\hcadre=0.15cm 
\setlength\unitlength{1.2mm} $\xymatrix@R=-1mm@C=0mm{
*{\GrTeXBox{\emptyset}}\arx1[dd]\arx1[rr] && *{\GrTeXBox{\btone}}\arx1[dd]\arx1[rr] && *{\GrTeXBox{\btr}}\arx1[dd]\arx1[rr] &&%
*{\GrTeXBox{\btlandr}}\arx1[dd]\arx1[rr] && *{\GrTeXBox{\btlandrl}}\arx1[dd]\arx1[rr] && *{\GrTeXBox{\btlrandrl}}\arx1[dd]\arx1[rr] &&%
*{\GrTeXBox{\btlrandrtolandr}}\arx1[dd]\\
& *{\GrTeXBox{}} && *{\GrTeXBox{}} && *{\GrTeXBox{}} &&%
*{\GrTeXBox{}} && *{\GrTeXBox{}} && *{\GrTeXBox{\permOneSymb}}\\
*{\GrTeXBox{\emptyset}}\arx1[dd]\arx1[rr] && *{\GrTeXBox{\btone}}\arx1[dd]\arx1[rr] && *{\GrTeXBox{\btr}}\arx1[dd]\arx1[rr] &&%
*{\GrTeXBox{\btlandr}}\arx1[dd]\arx1[rr] && *{\GrTeXBox{\btlandrl}}\arx1[dd]\arx1[rr] && *{\GrTeXBox{\btlrandrl}}\arx1[dd]\arx1[rr] &&%
*{\GrTeXBox{\btlrandrl}}\arx1[dd]\\
& *{\GrTeXBox{}} && *{\GrTeXBox{\permOneSymb}}\arx3[rrrrrrrruu] && *{\GrTeXBox{}} &&%
*{\GrTeXBox{}} && *{\GrTeXBox{}} && *{\GrTeXBox{}}\\
*{\GrTeXBox{\emptyset}}\arx1[dd]\arx1[rr] && *{\GrTeXBox{\btone}}\arx1[dd]\arx1[rr] && *{\GrTeXBox{\btone}}\arx1[dd]\arx1[rr] &&%
*{\GrTeXBox{\btl}}\arx1[dd]\arx1[rr] && *{\GrTeXBox{\btlandr}}\arx1[dd]\arx1[rr] && *{\GrTeXBox{\btlrandr}}\arx1[dd]\arx1[rr] &&%
*{\GrTeXBox{\btlrandr}}\arx1[dd]\\
& *{\GrTeXBox{}} && *{\GrTeXBox{}} && *{\GrTeXBox{}} &&%
*{\GrTeXBox{\permOneSymb}}\arx3[lllluu] && *{\GrTeXBox{}} && *{\GrTeXBox{}}\\
*{\GrTeXBox{\emptyset}}\arx1[dd]\arx1[rr] && *{\GrTeXBox{\btone}}\arx1[dd]\arx1[rr] && *{\GrTeXBox{\btone}}\arx1[dd]\arx1[rr] &&%
*{\GrTeXBox{\btl}}\arx1[dd]\arx1[rr] && *{\GrTeXBox{\btl}}\arx1[dd]\arx1[rr] && *{\GrTeXBox{\btlr}}\arx1[dd]\arx1[rr] &&%
*{\GrTeXBox{\btlr}}\arx1[dd]\\
& *{\GrTeXBox{\permOneSymb}}\arx3[uuuurr] && *{\GrTeXBox{}} && *{\GrTeXBox{}} &&%
*{\GrTeXBox{}} && *{\GrTeXBox{}} && *{\GrTeXBox{}}\\
*{\GrTeXBox{\emptyset}}\arx1[dd]\arx1[rr] && *{\GrTeXBox{\emptyset}}\arx1[dd]\arx1[rr] && *{\GrTeXBox{\emptyset}}\arx1[dd]\arx1[rr] &&%
*{\GrTeXBox{\btone}}\arx1[dd]\arx1[rr] && *{\GrTeXBox{\btone}}\arx1[dd]\arx1[rr] && *{\GrTeXBox{\btr}}\arx1[dd]\arx1[rr] &&%
*{\GrTeXBox{\btr}}\arx1[dd]\\
& *{\GrTeXBox{}} && *{\GrTeXBox{}} && *{\GrTeXBox{}} &&%
*{\GrTeXBox{}} && *{\GrTeXBox{\permOneSymb}} && *{\GrTeXBox{}}\\
*{\GrTeXBox{\emptyset}}\arx1[dd]\arx1[rr] && *{\GrTeXBox{\emptyset}}\arx1[dd]\arx1[rr] && *{\GrTeXBox{\emptyset}}\arx1[dd]\arx1[rr] &&%
*{\GrTeXBox{\btone}}\arx1[dd]\arx1[rr] && *{\GrTeXBox{\btone}}\arx1[dd]\arx1[rr] && *{\GrTeXBox{\btone}}\arx1[dd]\arx1[rr] &&%
*{\GrTeXBox{\btone}}\arx1[dd]\\
& *{\GrTeXBox{}} && *{\GrTeXBox{}} && *{\GrTeXBox{\permOneSymb}}\arx3[lllluuuu]\arx3[rrrruu] &&%
*{\GrTeXBox{}} && *{\GrTeXBox{}} && *{\GrTeXBox{}}\\
*{\GrTeXBox{\emptyset}}\arx1[rr] && *{\GrTeXBox{\emptyset}}\arx1[rr] && *{\GrTeXBox{\emptyset}}\arx1[rr] &&%
*{\GrTeXBox{\emptyset}}\arx1[rr] && *{\GrTeXBox{\emptyset}}\arx1[rr] && *{\GrTeXBox{\emptyset}}\arx1[rr] &&%
*{\GrTeXBox{\emptyset}}\\
 }$ } } \vspace{-0.5cm}
\begin{figure}[h]
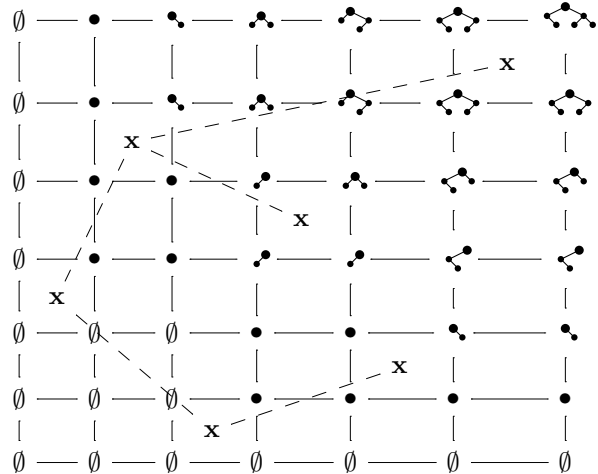

\caption{Example of growth diagram for the BST insertion algorithm.}
\end{figure}

\subsubsection{From a chain of binary trees to an increasing tree}\

Converting the horizontal chain $\fominQSylv$ into an increasing
tree is naturally done by labeling its trees according to the order
in which their appear in the path. \setlength\unitlength{1.4mm}
\def\mybygspace{\qquad}
{\small $$ \fominQSylv \mybygspace = \mybygspace \emptyset
\mybygspace \to \mybygspace \bullet \mybygspace \to \mybygspace \btr
\mybygspace \to \mybygspace \btlandr \mybygspace \to \mybygspace
\btlandrl \mybygspace \to \mybygspace \btlrandrl \mybygspace \to
\mybygspace \btlrandrtolandr$$}
\def\myspace{\,\,\,\,\,}
\def\myrayon{12pt}
{\small $$
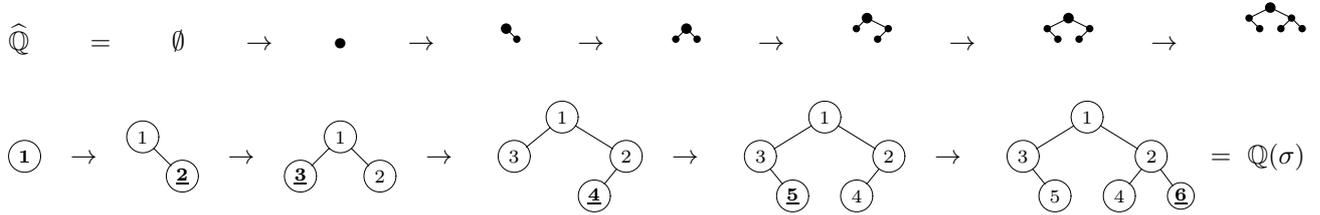
\begin{CD}
  \vcenter{\tiny\xymatrix@C=1mm@R=1mm{*=<\myrayon>[o][F-]{{\bf 1}}}}
  \myspace \to \myspace
  \vcenter{\tiny\xymatrix@C=1mm@R=1mm{ *=<\myrayon>[o][F-]{1}\ar@{-}[dr] \\ &
      *=<\myrayon>[o][F-]{{\bf \underline 2}} }}
  \myspace \to \myspace
\vcenter{\tiny\xymatrix@C=1mm@R=1mm{
&  *=<\myrayon>[o][F-]{1}\ar@{-}[dl]\ar@{-}[dr]\\
 *=<\myrayon>[o][F-]{{\bf \underline 3}} &&  *=<\myrayon>[o][F-]{2} }}
  \myspace \to \myspace
  \vcenter{\tiny\xymatrix@C=0mm@R=1mm{
      &    &&  *=<\myrayon>[o][F-]{1}\ar@{-}[dll]\ar@{-}[drr] \\
      & *=<\myrayon>[o][F-]{3} && && *=<\myrayon>[o][F-]{2}\ar@{-}[dl] \\
      &&  && *=<\myrayon>[o][F-]{{\bf \underline 4}}
    }}
  \myspace \to \myspace
  \vcenter{\tiny\xymatrix@C=0mm@R=1mm{
      &    &&  *=<\myrayon>[o][F-]{1}\ar@{-}[dll]\ar@{-}[drr] \\
      & *=<\myrayon>[o][F-]{3}\ar@{-}[dr] && && *=<\myrayon>[o][F-]{2}\ar@{-}[dl] \\
      &&  *=<\myrayon>[o][F-]{{\bf \underline 5}} && *=<\myrayon>[o][F-]{4}
    }}
  \myspace \to \myspace
  \vcenter{\tiny\xymatrix@C=0mm@R=1mm{
      &    &&  *=<\myrayon>[o][F-]{1}\ar@{-}[dll]\ar@{-}[drr] \\
      & *=<\myrayon>[o][F-]{3}\ar@{-}[dr] && && *=<\myrayon>[o][F-]{2}\ar@{-}[dl]\ar@{-}[dr] \\
      &&  *=<\myrayon>[o][F-]{5} && *=<\myrayon>[o][F-]{4} && *=<10pt>[o][F-]{{\bf \underline 6}}
    }} \,\, = \,\, \knuthitree(\sigma)
\end{CD}
$$}\

\subsubsection{From a chain of binary trees to a binary search
tree}\

In the vertical chain $\fominPSylv$, each tree $\fominPSylv_i \neq
\emptyset$ differs from its predecessor $\fominPSylv_{i-1}$ by the
node $\bullet_i$ that should be deleted and replaced by its own left
subtree (in any) in order to obtain $\fominPSylv_{i-1}$. Label this
node $\bullet_i$ with $i$, then the remaining nodes form a tree of
shape $\fominPSylv_{i-1}$ whose nodes labels where set during the
previous step of the conversion.

{\small
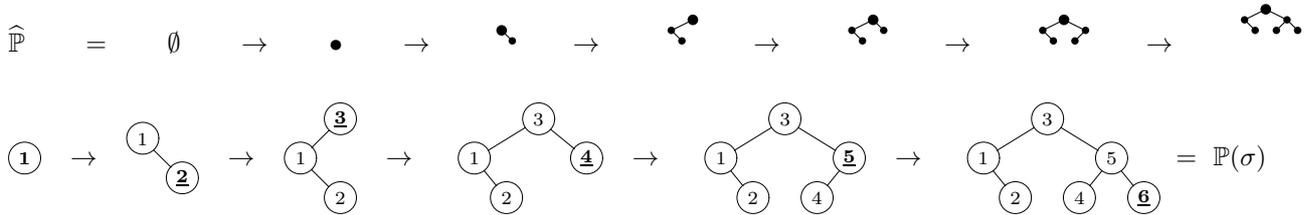
$$ \fominPSylv \mybygspace = \mybygspace \emptyset \mybygspace \to
\mybygspace \bullet \mybygspace \to \mybygspace \btr \mybygspace \to
\mybygspace \btlr \mybygspace \to \mybygspace \btlrandr \mybygspace
\to \mybygspace \btlrandrl \mybygspace \to \mybygspace
\btlrandrtolandr$$}
{\small $$
\begin{CD}
  \vcenter{\tiny\xymatrix@C=1mm@R=1mm{*=<\myrayon>[o][F-]{{\bf 1}}}}
  \myspace \to \myspace
  \vcenter{\tiny\xymatrix@C=1mm@R=1mm{ *=<\myrayon>[o][F-]{1}\ar@{-}[dr] \\ &
      *=<\myrayon>[o][F-]{{\bf \underline 2}} }}
  \myspace \to \myspace
\vcenter{\tiny\xymatrix@C=1mm@R=1mm{
&  *=<\myrayon>[o][F-]{{\bf \underline 3}}\ar@{-}[dl]\\
 *=<\myrayon>[o][F-]{1}\ar@{-}[dr]\\
 & *=<\myrayon>[o][F-]{2} }}
  \myspace \to \myspace
  \vcenter{\tiny\xymatrix@C=0mm@R=1mm{
      &    &&  *=<\myrayon>[o][F-]{3}\ar@{-}[dll]\ar@{-}[drr] \\
      & *=<\myrayon>[o][F-]{1}\ar@{-}[dr] && && *=<\myrayon>[o][F-]{{\bf \underline 4}} \\
      &&  *=<\myrayon>[o][F-]{2}
    }}
  \myspace \to \myspace
  \vcenter{\tiny\xymatrix@C=0mm@R=1mm{
      &    &&  *=<\myrayon>[o][F-]{3}\ar@{-}[dll]\ar@{-}[drr] \\
      & *=<\myrayon>[o][F-]{1}\ar@{-}[dr] && && *=<\myrayon>[o][F-]{{\bf \underline 5}}\ar@{-}[dl] \\
      &&  *=<\myrayon>[o][F-]{2} && *=<\myrayon>[o][F-]{4}
    }}
  \myspace \to \myspace
  \vcenter{\tiny\xymatrix@C=0mm@R=1mm{
      &    &&  *=<\myrayon>[o][F-]{3}\ar@{-}[dll]\ar@{-}[drr] \\
      & *=<\myrayon>[o][F-]{1}\ar@{-}[dr] && && *=<\myrayon>[o][F-]{5}\ar@{-}[dl]\ar@{-}[dr] \\
      &&  *=<\myrayon>[o][F-]{2} && *=<\myrayon>[o][F-]{4} && *=<\myrayon>[o][F-]{{\bf \underline 6}}
    }} \,\, = \,\, \knuthbtree(\sigma)
\end{CD}
$$}

\begin{proposition}For any permutation $\sigma$,
$\knuthbtree(\sigma) = \fominPSylv(\sigma)$ and $\knuthitree(\sigma)
= \fominQSylv(\sigma)$.
\end{proposition}

\begin{proof}
That building the growth diagram of a permutation $\sigma$ is a
parallel version of the application of the binary search tree
insertion algorithm to $\sigma$ follows from the remark that any
tree $\fominQSylv_k$ appearing in $\fominQSylv$ is the shape of the
binary search tree $\knuthbtree(\sigma_1 \sigma_2 \cdots \sigma_k)$
where $\sigma_1 \sigma_2 \cdots \sigma_k$ is the restriction of
$\sigma$ to its first $k$ letters, and any tree $\fominPSylv_k$
appearing in $\fominPSylv$ is the shape of
$\knuthbtree(\sigma_{/[1..k]})$ where $\sigma_{/[1..k]}$ is the
restriction of $\sigma$ to the interval $[1..k]$.
\end{proof}
\bibliographystyle{amsalpha}

\end{document}